\documentclass{amsart}
\usepackage{latexsym}

\usepackage{yfonts}

\newcommand{\RNum}[1]{\uppercase\expandafter{\romannumeral #1\relax}}
\newtheorem{theorem}{ Theorem}
\usepackage{amssymb,latexsym}

\usepackage{amsmath}
\newtheorem{lemma}{Lemma}
\newtheorem{definition}{Definition}
\newtheorem{corollary}{Corollary}

\newtheorem{proposition}{Proposition}
\begin{document}

\begin{center}
{\Large On some properties and relations between restricted barred preferential arrangements, multi-poly-Bernoulli numbers and related numbers}\\ 
 \vspace{5mm}
{\large S.Nkonkobe, \, V.Murali}\\
\vspace{2mm}\

{\it \footnotesize Department of Mathematics (Pure \& Applied)\\ Rhodes
University \\Grahamstown 6140 South Africa\\  snkonkobe@gmail.com \:\:v.murali@ru.ac.za }    \\
\vspace{2mm}
\end{center}

\section{Abstract\label{section:1}}

The introduction of bars in-between blocks of an ordered set partition(preferential arrangement) results in a barred ordered set partition(barred preferential arrangement). Having the restriction that some blocks of barred preferential arrangements to have a maximum of one block results in restricted barred preferential arrangements. 
In this study we establish relations between number of restricted barred preferential arrangements, multi-poly-Bernoulli numbers and numbers related to multi-poly-Bernoulli numbers. We prove a periodicity property satisfied by multi-poly-Bernoulli numbers having negative index, number of restricted barred preferential arrangements and numbers related to multi-poly-Bernoulli numbers having negative index. 
\\\\\fontsize{8}{0}\selectfont     
Mathematics Subject Classifications:05A18,05A19,05A16, 2013
\\\textbf{Keywords}: Barred preferential arrangements, Restricted barred preferential arrangements, multi-poly-Bernoulli numbers.\normalsize

\newpage
\section{introduction and preliminaries\label{intrductionandprelimenaries}}
Kaneko in \cite{Kaneko:first paper on poly-Bernoulli numbers} introduced poly-Bernoulli numbers defining them as\\ $\sum\limits_{n=0}^{\infty}B^k_n\frac{m^n}{n!}=\frac{Li_{k}(1-e^{-m)}}{1-e^{-m}}$; where $Li_k(m)$ is a poly-logarithm defined as \\$Li_k(m)=\sum\limits_{s=1}^{\infty}\frac{m^s}{s^k}$ where $k\in\mathbb{Z}$. Arakawa and Kaneko in \cite{first paper on multi-poly-bernoulli numbers} further generalised poly-Bernoulli numbers to multi-poly-Bernoulli numbers with the definition 

$\sum\limits_{n=0}^{\infty}B^{(j_1,\ldots,j_b)}_n=\frac{Li_{j_1\ldots j_b}(1-e^{-m)}}{(1-e^{-m})}$, where $Li_{j_1\ldots j_b}(m)=\sum\limits_{0<s_1<\cdots<s_b}\frac{m^{s_b}}{{s_1}^{j_1}\cdots{s_b}^{j_b}}$. 

The study of preferential arrangements seems to first appear in \cite{gross:1962}, although the integer sequence its self goes far as \cite{cayley:1859}. 
Introducing bars in-between blocks of a preferential arrangement forms a barred preferential arrangement\cite{barred:2013}. Recently the authors introduced the concept of restricted barred preferential arrangements by putting some restrictions on the sections of  barred preferential arrangements\cite{nkonkobe:Nelsen-Schmidt}.

\normalsize


\large \textbf{\underline{Barred preferential arrangements:}}
\\\\The concept of preferential arrangement of an $n$ element set was generalised by Pippenger et al in \cite{barred:2013} by introduction of bars in-between blocks of a preferential arrangement. Examples of barred preferential arrangements of $X_6$ with two and three bars are respectively
\\a)\: $|\;2\quad3\quad64|\quad1\quad5$
\\b) $ 6\;|\:3\:|1\quad24|5$

With reference to the bars, the barred preferential arrangement in a) has three sections, and the barred preferential arrangement in b) has four sections (see~\cite{barred:2013}).  
\\\\\large \textbf{\underline{Restricted barred preferential arrangements:}}\normalsize

In this study we view barred preferential arrangements as a result of first placing bars then distributing elements on the sections.\begin{definition}\cite{nkonkobe:Nelsen-Schmidt}
A section of a barred preferential arrangement is a restricted section if it can only have a maximum of one block.
\end{definition}
\begin{definition}\cite{nkonkobe:Nelsen-Schmidt}
A section of a barred preferential arrangement is a free section if elements distributed to the section can be preferential arranged in any possible way.  
\end{definition} 
A barred preferential arrangement of an $n$-element set in-which a number of fixed sections are restricted sections and other sections are free sections is referred to as a restricted barred preferential arrangement (see~\cite{nkonkobe:Nelsen-Schmidt}). We denote by $p^r_j(n)$ the total number of barred preferential arrangements of an $n$-element set having $k$ bars in-which $r$ fixed sections are restricted sections and the remaining $j=k+1-r$ sections are free sections. We denote the set of these barred preferential arrangements by $G^r_j(n)$, so $|G^r_j(n)|=p^r_j(n)$.

 For fixed $r,j\in\mathbb{N}_0=\{0,1,2,\ldots\}$ the number $p^r_j(n)$ of restricted barred preferential arrangements for $n\geq0$ is generated by (see~\cite{nkonkobe:Nelsen-Schmidt}); \begin{equation}\label{equation:1}P^r_j(m)=\frac{e^{rm}}{(2-e^m)^j}\qquad\qquad j,r\in\mathbb{N}_0\end{equation}  

For the case $j=1$ the above family of generating functions is the following Nelsen and Schmidt family of generating functions (see~\cite{Nelsen:91}).
\begin{equation}\label{equation:2}P^r_1(m)=\frac{e^{rm}}{2-e^m}\qquad\qquad r\in\mathbb{N}_0\end{equation} 

In this study we establish relations between  restricted barred preferential arrangements, multi-poly-Bernoulli numbers and some numbers related to multi-poly-Bernoulli numbers.  

\section{on some properties of restricted barred preferential arrangements}
          \begin{theorem}\label{theorem:3***}For $j\in\mathbb{N}_0$ and $n,r\in\mathbb{N}$ 
          \begin{center}
          $p^r_j(n)=\sum\limits_{s=0}^{n}\binom{n}{s}r^sp^0_j(n-s)$
          \end{center}\end{theorem}
         
         \begin{proof} On an element $\textgoth{W}\in G^r_j(n)$ we assume there are $s$ elements which are distributed among the $r$ restricted sections. The $s$ elements can be selected in $\binom{n}{s}$ ways. We can preferentially arrange the $s$ elements among the $r$ restricted sections in $r^s$ ways. We can then preferentially arrange remaining $n-s$ elements among the $j$ free sections in $p^0_j(n-s)$ ways. Taking the product and summing over $s$ we obtain the result.\end{proof}
          
         \begin{lemma}\label{lemma:1}\cite{gross:1962} For a  fixed $s\in\mathbb{N}_0$ and $n\geq1$ the following congruence holds;
         \begin{center}
         $s^{n+4}-s^{n}\:\equiv\:0\:mod\:10$
         \end{center} 
         
         \end{lemma}
         \begin{lemma}\cite{gross:1962}
                   For $n\geq1$ the last digit of the sequence $p^0_1(n)$  has a four cycle.
                   \end{lemma}
          \begin{theorem}\label{theorem:2***}For fixed $r,j\geq0$ such that $r>0$ or $j>0$ the last digit of the sequence $p^r_j(n)$  has a four cycle for $n\geq1$.
          \end{theorem}
      
     
    \begin{proof}\hspace*{2mm} $P^r_j(m)=\frac{e^{rm}}{(2-e^m)^j}
     \\\hspace*{25mm}=\frac{1}{2^j}\sum\limits_{s=0}^{\infty}\frac{\binom{-j}{s}(-1)^s e^{(r+s)m}}{2^s}$
     \\Hence \hspace*{6mm}$p^r_j(n)=[\frac{m^n}{n!}]P^r_j(m)=\frac{1}{2^j}\sum\limits_{s=0}^{\infty}\frac{\binom{-j}{s}(-1)^s (r+s)^n}{2^s}$.
     \\Letting $u=r+s$ we have
     \begin{equation}\label{equation:21B}p^r_j(n+4)-p^r_j(n)=\frac{1}{2^j}\sum\limits_{u=r}^{\infty}\frac{\binom{-j}{u-r}(-1)^{u-r}}{2^{u-r}}[u^{n+4}-u^n]\end{equation}
     Applying lemma~\ref{lemma:1} on \eqref{equation:21B} we obtain the result. \end{proof}

                           \begin{lemma} For $n\geq1$, $r\geq0$ 
                              \begin{center}
                              $p^r_1(n)=\sum\limits_{k=0}^{\infty}\sum\limits_{s=0}^{k}\binom{k}{s}(-1)^{s}(k-s+r)^n$
                              \end{center}
                              \end{lemma}
                              \begin{theorem}For $n,j\geq1$ and $r\geq0$ 
                              \begin{center}
                               $p^r_j(n)=\sum\limits_{k=0}^{\infty}\sum\limits_{s=0}^{k}\binom{k}{s}(-1)^{s}p^{r+k-s}_{j-1}(n)$
                              \end{center}
                              \end{theorem}

                              \begin{proof}The theorem is a generalisation of an un-labelled equation in \cite{gross:1962}.
                       
                                   $P^r_j(m)=\frac{e^{rm}}{(2-e^m)^j}=\frac{e^{rm}}{(2-e^m)^j}
                                   \\\hspace*{14mm}=\sum\limits_{k=0}^{\infty}\sum\limits_{s=0}^{\infty}\binom{k}{s}(-1)^s\frac{e^{(r+k-s)m}}{(2-e^m)^{j-1}}$.
                                  \\Hence $p^r_j(n)=[\frac{m^n}{n!}]P^r_j(m)=\sum\limits_{k=0}^{\infty}\sum\limits_{s=0}^{k}\binom{k}{s}(-1)^{s}p^{r+k-s}_{j-1}(n)$.\end{proof}
                                 
                                 \begin{theorem} For $r,j\geq1$ such that  $r\leq j$ we have 
               \begin{center}$p^r_{j-r}(n)=\sum\limits_{s=1}^{r}\binom{r}{s}(-1)^{s+1}\times p^s_{j-s}(n)$\end{center}
               \end{theorem}\begin{proof} On barred preferential arrangements having $j$ free sections, we fix $r\geq1$ sections. By the inclusion/exclusion principle the number of those barred preferential arrangements from $G^0_j(n)$ such that all the $r$ fixed sections have more than one block is \\$p^0_j(n)-\binom{r}{1}p^1_{j-1}(n)+\binom{r}{2}p^2_{j-2}(n)+\cdots+\binom{r}{r}p^r_{j-r}(n)(-1)^{r} =\sum\limits_{s=0}^{r}\binom{r}{s}p^s_{j-s}(n)(-1)^{s}$.
               Hence the number of those barred preferential arrangements such that all the $r$ fixed sections have a maximum of one block is $\sum\limits_{s=1}^{r}\binom{r}{s}p^s_{j-s}(n)(-1)^{s+1}=p^r_{j-r}(n)$. \end{proof}

           \begin{lemma}\cite{Combination locks paper}For $n\geq0$
           \begin{center}
           $p^0_1(n)=\sum\limits_{s=0}^{\infty}\frac{s^n}{2^{s+1}}$
           \end{center}
           \end{lemma}                       
          \begin{lemma}\cite{Nelsen:91}
            For $n\geq0$
           \begin{center}
$p^2_1(n)=2\sum\limits_{s=2}^{\infty}\frac{s^n}{2^s}$
 \end{center}
  \end{lemma}
  \begin{theorem}\label{theorem:6***} For $j\geq1$ and $n,r\geq0$ 
  \begin{center}
   $p^r_j(n)=\frac{1}{2}\sum\limits_{s=0}^{\infty}\frac{ p^{r+s}_{j-1}(n)}{2^s}$
   \end{center}
   \end{theorem}
   \begin{proof}
 $P^r_j(m)=\frac{e^{rm}}{(2-e^m)^j}
 \\\hspace*{21mm}=\frac{1}{2}\sum\limits_{s=0}^{\infty}
                                                     \frac{}{2^s}\frac{e^{r+s}m}{(2-e^m)^{j-1}}$
 \\Hence \hspace*{2mm}$p^r_j(n)=[\frac{m^n}{n!}]P^r_j(m)=\frac{1}{2}\sum\limits_{s=0}^{\infty}\frac{ p^{r+s}_{j-1}(n)}{2^s}$\end{proof}

  \begin{lemma}\cite{gross:1962} For $j,n\geq1$
   \begin{center}
  $p^0_1(n)=\sum\limits_{s=1}^{n-1}\binom{n}{s}p^0_1(n-s)+1$
   \end{center}
   \end{lemma}
   \begin{lemma}\cite{murali:combinatorics}\label{lemma:lemma1***} For $j,n\geq1$ 
   \begin{center}
   $p^2_1(n+1)=\sum\limits_{s=0}^{n}\binom{n+1}{s}p^2_1(s)+2^{n+1}$
    \end{center} 
   \end{lemma}
  
  \begin{theorem} For $j,n\geq1$
 \begin{center}
  $p^r_j(n)=p^{r}_{j-1}(n)+\sum\limits_{s=0}^{n-1}\binom{n}{s}p^{r}_{j}(s) $
  \end{center}
   \end{theorem}
     \begin{proof}                                 
  The theorem is a generalisation of (9) of \cite{gross:1962}.
 \\ By theorem~\ref{theorem:6***} we have \begin{equation}\label{equation:2***}
   p^r_j(n)=\frac{1}{2}\sum\limits_{s=0}^{\infty}\frac{ p^{r+s}_{j-1}(n)}{2^s}
   \end{equation}
                                  
                                  
  This implies that
  \begin{equation}\label{equation:3***}\sum\limits_{m=0}^{n-1}\binom{n}{m}p^r_j(n-m)=\frac{1}{2}\sum\limits_{s=0}^{\infty}\begin{bmatrix}{\sum\limits_{m=0}^{n}\binom{n}{m}p^{r+s}_{j-1}(n-m)\times1^{s}-1}\end{bmatrix}\frac{1}{2^s}\end{equation}

                                  
   So  
   $$\sum\limits_{m=0}^{n-1}\binom{n}{m}p^r_j(n-m)=\frac{1}{2}\sum\limits_{s=0}^{\infty}\frac{p^{r+s+1}_{j-1}(n)}{2^s}-1$$
                                  
   Letting $s+1=k$ and applying \eqref{equation:2***} we obtain
                                  $$p^r_j(n)=\sum\limits_{m=1}^{n}\binom{n}{m}p^r_j(n-m)+p^r_{j-1}(n)$$
                                  \end{proof}

\section{poly-Bernoulli numbers\label{section:poly-Bernoulli number}}

\begin{proposition}\cite{gereneratingfunctionology:1994}\label{proposition:1}
A formal power series $p(m)=\sum\limits_{n=0}^{\infty}c_n\times m^n$ has a reciprocal if $c_0\not=0$.
\end{proposition}

The $n^{th}$ term of the reciprocal $\frac{1}{p(m)}=\sum\limits_{n=0}^{\infty}c^*_n\times m^n$  when it exists is given by (see~\cite{gereneratingfunctionology:1994}) 
\begin{equation}\label{equation:4}
c^*_n=\frac{-1}{c_0}\sum\limits_{s=1}^n c_s c^*_{n-s} \qquad where\quad c^*_0=\frac{1}{c_0}
\end{equation}


 A closed form for the poly-Bernoulli numbers $B^{-2}_n$ (see~\cite{Kamano:main theorem paper on multi-poly-Bernoulli numbers})

\begin{equation}\label{equation:3}
B^{-2}_n=2\times3^n-2^n \quad where\quad n\in\mathbb{N}_0
\end{equation}


 \begin{lemma}For $n\geq0$ and fixed  $j\in\mathbb{Z}$
 \begin{center}
 $B^j_n=\sum\limits_{n=0}^{\infty}\sum\limits_{s=0}^{n}\frac{1}{(s+1)^j}\sum\limits_{i=0}^{s}\binom{s}{i}(-1)^{s-i}(i-s)^n$
 \end{center}
 \end{lemma}
 
\begin{proof} By definition $\sum\limits_{n=0}^{\infty} B^{j}_n\frac{m^n}{n!}=\frac{Li_{j}(1-e^{-m})}{(1-e^{m})}$
 $\implies\sum\limits_{n=0}^{\infty} B^{j}_n\frac{m^n}{n!}=\sum\limits_{s=0}^{\infty}\frac{(1-e^{-m})^{s}}{{(s+1)}^j}$.
 \\From this it follows that $B^{j}_n=\sum\limits_{s=0}^{\infty}\frac{1}{(s+1)^j}\sum\limits_{i=0}^{s}\binom{s}{i}(-1)^{s-i}\sum\limits_{n=0}^{\infty}(i-s)^n$\end{proof}

\vspace*{5mm}
We recall the family of generating functions for number of restricted barred preferential arrangements for $j=1$ is $P^r_1(m)=\frac{e^{rm}}{2-e^m}\quad (\text{where}\:\: r\in\mathbb{N}_0$). We denote by $P^r_1(m)^*$ the reciprocal of the generating function $P^r_1(m)$; we denote as $P^r_1(m)^*=\sum\limits_{n=0}^{\infty}\frac{a^r_1(n)\times m^n}{n!}$. 
\\\\We first consider  $P^3_1(m)^*=\frac{2-e^m}{e^{3m}}=\sum\limits_{n=0}^{\infty}\frac{a^3_1(n)\times m^n}{n!}$. So $a^3_1(n)=(-1)^n(2\times 3^n-2^n)$. $\implies$ $|a^3_1(n)|=2\times 3^n-2^n=B^{-2}_n$ (by \eqref{equation:3}).

By \eqref{equation:4} we have;
\begin{equation}\label{equation:6}
p^3_1(n)=\sum\limits_{s=1}^{n}\binom{n}{s}(-1)^{s+1}B^{-2}_s\times p^3_1(n-s) \quad for\quad n\geq1
\end{equation}

From the generating functions we deduce that for $r\geq3$, $j\geq1$
\begin{equation}\label{equation:5}
p^{r-3}_{j-1}(n)=\sum\limits_{s=0}^{n}\binom{n}{s}(-1)^sB^{-2}_s\times p^r_j(n-s)
\end{equation}

\vspace*{10mm}
\begin{flushleft}
\large \textbf{\underline{Multi-poly-Bernoulli numbers:}}\end{flushleft}\normalsize
\vspace*{3mm}
\begin{theorem}\label{theorem:8}\cite{Kamano:main theorem paper on multi-poly-Bernoulli numbers}For $n\geq0$ we have
\\$\sum\limits_{j_1=0}^{\infty}\sum\limits_{j_2=0}^{\infty}\cdots\sum\limits_{j_b=0}^{\infty}\times\sum\limits_{n=0}^{\infty}B^{(-j_1,\ldots,-j_b)}_n\frac{r^{j_1}_1}{j_1!}\frac{r^{j_2}_2}{j_2!}\cdots\frac{r^{j_b}_b}{j_b!}\frac{m^n}{n!}$\\$=\frac{1}{(e^{-r_1-r_2\cdots-r_b}+e^{-m}-1)(e^{-r_2\cdots-r_b}+e^{-m}-1)\cdots(e^{-r_b}+e^{-m}-1)}$
\end{theorem}
 On theorem~\ref{theorem:8} when we let $r_2=r_3=\cdots=r_b=0$ we obtain
\begin{equation}\label{equation:10}\sum\limits_{j=0}^{\infty}\sum\limits_{n=0}^{\infty}B^{(-j,\overset{b-1}{\overbrace{0,0,\cdots,0})}}_n\frac{r^j}{j!}\frac{m^n}{n!}=\frac{e^{(b-1)m}}{e^{-r}+e^{-m}-1}=\begin{pmatrix}e^{(b-1)m}\end{pmatrix}\begin{pmatrix}\frac{1}{e^{-r}+e^{-m}-1}  \end{pmatrix}\end{equation} 
\begin{corollary}For fixed $b\in\mathbb{N}$ and $j\in\mathbb{N}_0$
\begin{center}$B^{(-j,\overset{b-1}{\overbrace{0,0,\cdots,0})}}_n=\sum\limits_{s=0}^n\binom{n}{s}B^{(\overset{b-1}{\overbrace{0,0,\cdots,0})}}_s\times B^{-j}_{n-s}$\end{center}
\end{corollary}

\begin{theorem}\cite{Kamano:main theorem paper on multi-poly-Bernoulli numbers}\label{theorem:4} For a fixed $b\in\mathbb{N}$ and $j_1,j_2,\ldots,j_b\in\mathbb{N}_0$ such that $(j_1,j_2,\ldots,j_b)\not=(0,0,\ldots,0)$. Let $j=j_1+j_2+\cdots+j_b$. Then the following identity holds 
\begin{center}
$B^{(-j_1,\ldots,-j_b)}_n=\sum\limits_{s=1}^{j}\mu_s^{(j_1,\ldots,j_b)}(s+b)^n$
\end{center}
Where $\mu_s^{(j_1,\ldots,j_b)}$ are integers recursively defined in the following way
\\\RNum{1}. $\mu^{(j_1)}_{s}=(-1)^{s+j_1}s!{j_1\brace s}$
\\\RNum{2}.  $\mu_s^{(j_1,\ldots,j_{b-1},0)}=\mu_s^{(j_1,\ldots,j_{b-1})}$
\\\RNum{3}. $\mu_s^{(j_1,\ldots,j_{b-1},j_{b}+1)}=(s+b-1)\mu_{s-1}^{(j_1,\ldots,j_{b-1},j_{b})}-s\times\mu_s^{(j_1,\ldots,j_{b-1},j_{b})}$ 
\\Where $\mu^{(j_1,\ldots,j_{b-1},j_{b})}_0=\begin{cases} 1&\:\: if\:\: (j_1,\ldots,j_{b-1},j_{b})=(0,0,\ldots,0)\\
0&\:\: Otherwise 
\end{cases}$

and $\mu^{(j_1,\ldots,j_{b-1},j_{b})}_s=0$ for all $s>j$.
\end{theorem}
\vspace{5mm}
Form the theorem we have;
$$B^{(-2)}_n=2\times3^n-2^n$$
$$B^{(-2,0)}_n=2\times4^n-3^n$$
$$B^{(-2,0,0)}_n=2\times5^n-4^n$$
$$B^{(-2,0,0,0)}_n=2\times6^n-5^n$$

Inductively,
\begin{equation}\label{equation:7}B^{(-2,\overset{b}{\overbrace{0,0,\cdots,0})}}_n=2\times(3+b)^n-(3+b-1)^n\qquad where\:\: b\in\mathbb{N}_0\end{equation}
We write the generating functions for number of restricted barred preferential arrangements for $j=1$, $r\geq3$  as 
 \begin{equation}\label{equation:13B}P^{r+b}_1(m)=\frac{e^{(3+b)m}}{2-e^m} \quad\text{where}\:\:\: b\in\mathbb{N}_0\end{equation} So $P^r_1(m)^*=\frac{2-e^m}{e^{(3+b)m}}$.  
$\implies[\frac{m^n}{n!}]P^r_1(m)^*= (-1)^n[2\times(3+b)^n-(3+b-1)^n]\\\implies|[\frac{m^n}{n!}]P^r_1(m)^*|=B^{(-2,\overset{b}{\overbrace{0,0,\cdots,0})}}_n$\hspace{3mm}(by \eqref{equation:7}).


\vspace{2mm}
So by \eqref{equation:4} we have

\begin{equation}\label{equation:8}
p^{3+b}_{1}(n)=\sum\limits_{s=1}^{n}\binom{n}{s}(-1)^{s+1}B^{(-2,\overset{b}{\overbrace{0,0,\cdots,0})}}_s\times p^{3+b}_{1}(n-s)\quad where\:\: b\in\mathbb{N}_0
\end{equation}

Where $p^{3+b}_{1}(n)$ denotes number of restricted barred preferential arrangements. 
The result in \eqref{equation:8} can equivalently be written as;
$$B^{(-2,\overset{b}{\overbrace{0,0,\cdots,0})}}_n=\sum\limits_{s=1}^{n}\binom{n}{s}p^{3+b}_1(s)\times(-1)^{n-s+1}B^{(-2,\overset{b}{\overbrace{0,0,\cdots,0})}}_{n-s}$$

On a convolution of $P^r_j(m)=\sum\limits_{n=0}^{\infty}\frac{p^r_j(n)\times m^n}{n!}$  and  \\$P^r_1(m)^*=\frac{2-e^m}{e^{(3+b)m}}=\sum\limits_{n=0}^{\infty}(-1)^n B^{(-2,\overset{b}{\overbrace{0,0,\cdots,0})}}_{n}\frac{m^n}{n!}$ we obtain;

\begin{equation}\label{equation:9}
p^{r-(3+b)}_{j-1}(n)=\sum\limits_{s=0}^{n}\binom{n}{s}p^r_j(s)\times(-1)^{n-s}\times B^{(-2,\overset{b}{\overbrace{0,0,\cdots,0})}}_{n-s}
\hspace*{0mm}\:\:\:\: \text{for}\:\: r\geq3+b, j\geq1\end{equation}

\begin{lemma}\label{lemma:2B}For $n\geq1$ and fixed $b\in\mathbb{N}_0$ the last digit of the sequence 

$B^{(-2,\overset{b}{\overbrace{0,0,\cdots,0})}}_n$ has a four cycle.
\end{lemma}

\begin{proof}By \eqref{equation:7} we have $B^{(-2,\overset{b}{\overbrace{0,0,\cdots,0})}}_n=2\times(3+b)^n-(3+b-1)^n$.
So $B^{(-2,\overset{b}{\overbrace{0,0,\cdots,0})}}_{n+4}-B^{(-2,\overset{b}{\overbrace{0,0,\cdots,0})}}_n=2[(3+b)^{n+4}-(3+b)^n] -[(2+b)^{n+4}-(2+b)^n]$. By lemma~\ref{lemma:1} both $[(3+b)^{n+4}-(3+b)^n]$ and $[(2+b)^{n+4}-(2+b)^n]$ are divisible by 10.\end{proof}
\begin{theorem}\label{theorem:2B}
For fixed $j_1,j_2,\ldots,j_b\in\mathbb{N}_0$ the last digit of the sequence
\\$B^{(-j_1,\ldots,-j_b)}_n$ for $n\geq1$ has a four cycle.
\end{theorem}
\begin{proof}By definition $\sum\limits_{n=0}^{\infty} B^{(-j_1,\ldots,-j_b)}_n\frac{m^n}{n!}=\frac{Li_{-j_1,\ldots,-j_b}(1-e^{-m})}{(1-e^{m})^b}=\sum\limits_{0<s_1<s_2<\cdots<s_b}{s_1}^{j_1}\times {s_2}^{j_2}\times\cdots\times {s_b}^{j_b}(1-e^{-m})^{s_b-b}$

$\implies
B^{(-j_1,\ldots,-j_b)}_n=\sum\limits_{0<s_1<s_2<\cdots<s_b}{s_1}^{j_1}\times\cdots\times {s_b}^{j_b}\times\\(-1)^{s_b-b}\sum\limits_{i=0}^{s_b-b}\binom{s_b-b}{i}(-1)^{s_b-b-1}(-1)^n i^n$.
\\$\implies$ $B^{(-j_1,\ldots,-j_b)}_{n+4}-B^{(-j_1,\ldots,-j_b)}_n=\sum\limits_{0<s_1<s_2<\cdots<s_b}{s_1}^{j_1}\times\\ {s_2}^{j_2}\times\cdots\times {s_b}^{j_b}(-1)^{s_b-b}\sum\limits_{i=0}^{s_b-b}\binom{s_b-b}{i}(-1)^{s_b-b-1}(-1)^n [i^{n+4}-i^n]$.
\\ By applying lemma~\ref{lemma:1} we obtain the result.\end{proof}

\begin{theorem} For fixed $b\in\mathbb{N}_0$ we consider barred preferential arrangements of $X_n$ having $3+b$ bars where all the sections are restricted sections. For fixed sections the $i^{th}$ and the $j^{th}$ the poly-Bernoulli number 
$B^{(-2,\overset{b}{\overbrace{0,0,\cdots,0})}}_n$ is the number of restricted barred preferential arrangements such that the $i^{th}$ or $j^{th}$ section is empty. 
\end{theorem}
\begin{proof}We consider restricted barred preferential arrangements of $X_n$ having $3+b$ bars where all the sections are restricted sections. We fix two sections (the $i^{th}$ and the $j^{th}$ sections). The number of those restricted barred preferential arrangements whose $i^{th}$ section is empty is $(3+b)^n$. The number of those restricted barred preferential arrangements whose $j^{th}$ section is empty is also $(3+b)^n$. The number of those restricted barred preferential arrangements whose $i^{th}$ and $j^{th}$ section are empty is $((3+b)-1)^n$.  
By the inclusion/exclusion principle the number of restricted barred preferential arrangements whose $i^{th}$ or $j^{th}$ sections is empty, is $2\times (3+b)^{n}-(3+b-1)^{n}$ where $b\geq-2$. 
Hence on \eqref{equation:7} the number $2\times(3+b)^n-(3+b-1)^n=B^{(-2,\overset{b}{\overbrace{0,0,\cdots,0})}}_n$ is the number of restricted barred preferential arrangements of $X_n$ having $3+b$ bars; where all the sections are restricted sections such that the $i^{th}$ or $j^{th}$ section is empty.  \end{proof}

\section{On some related numbers}
 We define numbers $U^{(j_1,\ldots,j_b)}_n$ by the generating function \\ 
 $ \sum\limits_{n=0}^{\infty} U^{(j_1,\ldots,j_b)}_n\frac{m^n}{n!}=\frac{Li_{j_1,\ldots,j_b}(1-e^{-m})}{(1-e^{m})^b}e^{-m}$ (For the case $b=1$ the numbers appears~in~\cite{first paper on multi-poly-bernoulli numbers})

  By section~\ref{section:poly-Bernoulli number} we have 
  \begin{equation}\label{equation:11B}\sum\limits_{n=0}^{\infty}B^{(-2,\overset{b}{\overbrace{0,0,\cdots,0})}}_n\frac{m^n}{n!}=\frac{2-e^m}{e^{(3+b)m}}\end{equation} By definition of $U^{(j_1,\ldots,j_b)}_n$ we have \begin{equation}\label{equation:12B}\sum\limits_{n=0}^{\infty}U^{(-2,\overset{b}{\overbrace{0,0,\cdots,0})}}_n\frac{m^n}{n!}=\frac{2-e^m}{e^{(3+b+1)m}}\end{equation}
  
  By  \eqref{equation:13B} and \eqref{equation:12B} we have
    \begin{equation}
     p^{3+b+1}_{1}(n)=\sum\limits_{s=1}^{n}\binom{n}{s}(-1)^{s+1}U^{(-2,\overset{b}{\overbrace{0,0,\cdots,0})}}_s\times p^{3+b+1}_{1}(n-s)
     \end{equation}

  From \eqref{equation:11B} we deduce that  $U^{(-2,\overset{b}{\overbrace{0,0,\cdots,0})}}_n=B^{(-2,\overset{b+1}{\overbrace{0,0,\cdots,0})}}_n$. Hence for fixed $b\in\mathbb{N}_0$ the sequence $U^{(-2,\overset{b}{\overbrace{0,0,\cdots,0})}}_n$ for $n\geq1$ has a four cycle (by lemma~\ref{lemma:2B}).

  \begin{theorem}For fixed $j_1,j_2,\ldots,j_b\in\mathbb{N}_0$ the last digit of the sequence
  \\$U^{(-j_1,\ldots,-j_b)}_n$ for $n\geq1$ has a four cycle.
  \end{theorem}
 \begin{proof}By definition $ \sum\limits_{n=0}^{\infty} U^{(-j_1,\ldots,-j_b)}_n\frac{m^n}{n!}=\frac{Li_{-j_1,\ldots,-j_b}(1-e^{-m})}{(1-e^{m})^b}e^{-m}
\\\hspace*{0mm}=\sum\limits_{0<s_1<s_2<\cdots<s_b}{s_1}^{j_1}\times\cdots\times {s_b}^{j_b}(1-e^{-m})^{s_b-b}e^{-m}
$
This implies that
 \\\hspace*{0mm} $U^{(-j_1,\ldots,-j_b)}_n=\sum\limits_{0<s_1<s_2<\cdots<s_b}{s_1}^{j_1}\times\cdots\times {s_b}^{j_b}\frac{(-1)^{s_b-b+1}}{s_b-b+1}\sum\limits_{i=0}^{s_b-b+1}\binom{s_b-b+1}{i}(-1)^{s_b-b+1-i}\times\\\sum\limits_{n=0}^{\infty}(-i)^{n+1}=\sum\limits_{n=0}^{\infty}\sum\limits_{0<s_1<s_2<\cdots<s_b}{s_1}^{j_1}\times\cdots\times {s_b}^{j_b}\frac{(-1)^{s_b-b+1}}{s_b-b+1}\sum\limits_{i=0}^{s_b-b+1}\binom{s_b-b+1}{i}\times\\(-1)^{s_b-b+1-i}(-i)^{n+1}$. This implies that
  \\$U^{(-j_1,\ldots,-j_b)}_{n+4}-U^{(-j_1,\ldots,-j_b)}_n=\sum\limits_{n=0}^{\infty}\sum\limits_{0<s_1<s_2<\cdots<s_b}{s_1}^{j_1}\times\cdots\times {s_b}^{j_b}\frac{(-1)^{s_b-b+1}}{s_b-b+1}\times\\\sum\limits_{i=0}^{s_b-b+1}\binom{s_b-b+1}{i}(-1)^{s_b-b+1-i}[(-i)^{n+1+4}-(-i)^{n+1}]$.
  By lemma~\ref{lemma:1} the sequence $U^{(-j_1,\ldots,-j_b)}_n$ for $n\geq1$ has a four cycle.\end{proof}
  
  \begin{theorem}\label{theorem:3B}For fixed $j_1,j_2,\ldots,j_b\in\mathbb{Z}$
  \begin{center}
 $U^{(j_1,j_2,\ldots,j_b)}_n=(-1)^{n+1}\sum\limits_{s_b=b}^{n+b}\;\sum\limits_{0<s_1<\cdots<s_b}\frac{1}{s^{j_1}_1\times\cdots\times s^{j_b}_b}(-1)^{s_b-b+1}\times(s_b-b)!{{n+1}\brace{s_b-b+1}}$
  \end{center}
  \end{theorem}
  
  \begin{proof}By definition $\sum\limits_{n=0}^{\infty} U^{(j_1,\ldots,j_b)}_n\frac{m^n}{n!}=\frac{Li_{j_1,\ldots,j_b}(1-e^{-m})}{(1-e^{-m})^b}e^{-m}\\\implies\sum\limits_{n=0}^{\infty} U^{(j_1,\ldots,j_b)}_n\frac{m^n}{n!}=\sum\limits_{0<s_1<s_2<\cdots<s_b}\frac{(1-e^{-m})^{s_b-b}e^{-m}}{{s_1}^{j_1}\times {s_2}^{j_2}\times\cdots\times {s_b}^{j_b}}
  \\=\sum\limits_{0<s_1<s_2<\cdots<s_b}\frac{1}{{s_1}^{j_1}\times {s_2}^{j_2}\times\cdots\times {s_b}^{j_b}}\frac{d}{dm}\frac{(1-e^{-m})^{s_b-b+1}}{s_b-b+1}$.
  \\Now  applying the identity $\sum\limits_{n=s}^{\infty}{n\brace s}=\frac{(e^m-1)^s}{s!}$ (see~\cite{book on Bernoulli numbers},pp~32)
\\ We have $\sum\limits_{n=0}^{\infty} U^{(j_1,\ldots,j_b)}_n\frac{m^n}{n!} 
\\=\sum\limits_{0<s_1<s_2<\cdots<s_b}\frac{(-1)^{s_b-b+1}}{{s_1}^{j_1}\times {s_2}^{j_2}\times\cdots\times {s_b}^{j_b}}\sum\limits_{n=s_b-b}(-1)^{n+1}{n+1\brace {s_b-b+1}}(s_b-b)!\times\frac{m^n}{n!}
\\=\sum\limits_{n=0}^{\infty}\begin{pmatrix}(-1)^{n+1}\sum\limits_{s_b=b}^{n+b}\sum\limits_{0<s_1<s_2<\cdots<s_b}\frac{(-1)^{s_b-b+1}}{{s_1}^{j_1}\times {s_2}^{j_2}\times\cdots\times {s_b}^{j_b}}{n+1\brace {s_b-b+1}}(s_b-b)!\end{pmatrix}\frac{m^n}{n!}$.\\

\end{proof}
The theorem is an analogue of theorem~7 of \cite{closed form for multi-poly-Bernoulli numbers}
\\\\\\\Large{\textbf{\underline{Acknowledgements:}}}\normalsize

The first author would like to thank the National research foundation of South Africa for financial support under the NRF-Innovation doctoral scholarship and Rhodes University for support under the Levenstein bursary.
 Both authors acknowledge support from Rhodes University.

\end{document}